\documentclass[a4paper]{amsart} 
\usepackage{amssymb} 
\usepackage[all]{xy}
\SelectTips{cm}{}

\usepackage{hyperref}
\calclayout

\title{A note on thick subcategories of stable derived categories}

\thanks{Version from November 9, 2011.}

\author{Henning Krause}
\address{Henning Krause\\ Fakult\"at f\"ur Mathematik\\
Universit\"at Bielefeld\\ D-33501 Bielefeld\\ Germany.}
\email{hkrause@math.uni-bielefeld.de}

\author{Greg Stevenson}
\address{Greg Stevenson\\ Fakult\"at f\"ur Mathematik\\
Universit\"at Bielefeld\\ D-33501 Bielefeld\\ Germany.}
\email{gstevens@math.uni-bielefeld.de}

\newtheorem{lem}{Lemma}

\newtheorem{cor}[lem]{Corollary}
\newtheorem{thm}[lem]{Theorem}

\theoremstyle{remark}
\newtheorem{rem}[lem]{Remark}

\theoremstyle{definition}

\numberwithin{equation}{section}

\renewcommand{\mod}{\operatorname{\mathsf{mod}}\nolimits}

\newcommand{\Proj}{\operatorname{\mathsf{Proj}}\nolimits}
\newcommand{\agProj}{\operatorname{{Proj}}\nolimits}
\newcommand{\proj}{\operatorname{\mathsf{proj}}\nolimits}

\newcommand{\Inj}{\operatorname{\mathsf{Inj}}\nolimits}

\newcommand{\Hom}{\operatorname{Hom}\nolimits}

\newcommand{\Ker}{\operatorname{Ker}\nolimits}
\newcommand{\Coker}{\operatorname{Coker}\nolimits}
\newcommand{\coh}{\operatorname{coh}\nolimits}

\newcommand{\Ext}{\operatorname{Ext}\nolimits}

\newcommand{\Spec}{\operatorname{Spec}\nolimits}

\newcommand{\stmod}{\operatorname{\mathsf{stmod}}\nolimits}

\newcommand{\Sing}{\operatorname{Sing}\nolimits}

\newcommand{\xto}{\xrightarrow}

\def\a{\alpha}

\def\d{\delta}

\def\A{{\mathsf A}}

\def\C{{\mathsf C}}
\def\D{{\mathsf D}}

\def\K{{\mathsf K}}

\def\P{{\mathsf P}}

\def\str{{\mathcal O}}

\def\bfi{{\mathbf i}}

\def\bbZ{{\mathbb Z}}
\def\bbX{{\mathbb X}}

\def\frm{{\mathfrak m}}
\def\frp{{\mathfrak p}}

\begin{document}

\begin{abstract}
  For an exact category having enough projective objects, we establish
  a bijection between thick subcategories containing the projective
  objects and thick subcategories of the stable derived category. Using 
  this bijection we classify thick subcategories of finitely generated 
  modules over local complete intersections and produce generators 
  for the category of coherent sheaves on a separated noetherian 
   scheme with an ample family.
\end{abstract}

\maketitle
\tableofcontents

\section{Introduction}

Let $A$ be a not necessarily commutative ring and $\mod A$ be the
category of finitely presented $A$-modules.  Consider the \emph{stable
  derived category} of $A$ in the sense of Buchweitz \cite{Buc}, which
is also called the \emph{triangulated category of singularities} or just the \emph{singularity category},
following work of Orlov \cite{Orl}. This category is by definition the
Verdier quotient of the bounded derived category of $\mod A$ with
respect to the triangulated subcategory consisting of all perfect
complexes: \[\D^b(\mod A)/\D^b(\proj A).\]

In a number of recent papers, thick subcategories of this triangulated
category have been studied and even classified, typically in terms of
primes ideals of some appropriate cohomology ring \cite{BCR,
  BIK,OpSt,Ste,Tak}. In this note we point out a bijection between
thick subcategories of the stable derived category and thick
subcategories of the module category containing all projective
modules. In a special case this bijection has been observed by
Takahashi \cite{Tak:2010}. Using our more general form of this 
bijection we are able to extend Takahashi's Theorem 4.6 (1) to 
complete intersections. 

Further applications of this bijection arise from the study of generators of
exact and triangulated categories. We illustrate this  by
results of Oppermann--Stovicek \cite{OpSt} and Schoutens
\cite{Sch}. In the latter case, we include a generalisation 
as well as a new proof.

\section{Thick subcategories versus thick subcategories}

Let $\A$ be an exact category in the sense of Quillen \cite{Qui} and
denote by $\D^b(\A)$ its bounded derived category \cite{Nee,Ver}.
Suppose that $\A$ has enough projective objects. Let $\Proj \A$ be the
full subcategory consisting of the projective objects and view
$\D^b(\Proj\A)$ as a thick subcategory of $\D^b(\A)$. The Verdier
quotient
\[\D^b(\A)/\D^b(\Proj\A)\]
is by definition the \emph{stable derived category} of $\A$.

We are interested in thick subcategories of the stable derived
category and observe that they correspond bijectively to thick
subcategories of $\D^b(\A)$ containing all projective objects. Recall
that a full additive subcategory of a triangulated category is
\emph{thick} if it is closed under shifts, mapping cones, and direct
summands.

A full additive subcategory $\C$ of $\A$ is called \emph{thick} if it
is closed under direct summands and has the following two out of three
property: for every exact sequence $0\to X\to Y\to Z\to 0$ in $\A$
with two terms in $\C$, the third term belongs to $\C$ as well.

In the following we identify $\A$ with the full subcategory of $\D^b(\A)$ consisting of
all complexes concentrated in degree zero.

\begin{thm}\label{th:thick}
  Let $\A$ be an exact category having enough projective objects.  The
  map sending a subcategory $\D$ of $\D^b(\A)$ to $\D\cap \A$ induces
  a bijection between
\begin{enumerate} 
\item[--] the thick subcategories of $\D^b(\A)$ containing all
  projective objects, and 
\item[--] the thick subcategories of $\A$ containing all projective
  objects.
\end{enumerate} 
The inverse map sends a thick subcategory $\C$ of $\A$ to $\D^b(\C)$.
\end{thm}

\begin{proof}
We may assume that $\A$ is idempotent complete, keeping in mind that
the idempotent completion of $\D^b(\A)$ equals the bounded derived
category of the idempotent completion of $\A$; see \cite{BaSc}.

The first part of this proof is taken from the appendix of
\cite{BeKr}.  Suppose that $\C$ is a thick subcategory of $\A$
containing all projectives.  The category $\C$ inherits an exact
structure from $\A$ and the inclusion $\C\to\A$ induces therefore an
exact functor $\D^b(\C)\to\D^b(\A)$.  The fact that $\C$ contains the
projective objects implies that this functor is fully faithful; see
for instance \cite[Proposition~III.2.4.1]{Ver}. Thus the full
subcategory $\D$ of $\D^b(\A)$ consisting of objects isomorphic to a
complex of objects in $\C$ is a thick subcategory.  We claim that
$\C=\D\cap\A$. Clearly, $\C\subseteq\D\cap\A$. Thus we fix
$X\in\D\cap\A$. Then $X$ is in $\D^b(\A)$ isomorphic to a bounded
complex $C$ with differential $\d$ such that $C^n\in\C$ for all $n$
and $C$ is acyclic in all degrees $n\neq 0$. Now we use that $\C$ is
thick. Thus $\Coker\d^{-2}$ and $\Ker\d^0$ belong to $\C$, and we have
an admissible monomorphism $\Coker\d^{-2}\to\Ker\d^0$ such that the
cokernel is isomorphic to $X$. We conclude that $X$ belongs to $\C$,
and therefore $\C=\D\cap\A$.

  Now fix a thick subcategory $\D$ of $\D^b(\A)$ containing all
  projectives and set $\C=\D\cap\A$.  Each exact sequence $0\to X\to
  Y\to Z\to 0$ in $\A$ gives rise to an exact triangle $X\to Y\to Z\to
  X[1]$ in $\D^b(\A)$. Thus $\C$ is a thick subcategory of $\A$. We
  claim that the functor $\D^b(\C)\to\D$ is an equivalence. In view of
  the first part of the proof, it suffices to show that each object
  $X$ in $\D$ is isomorphic to a complex of objects in $\C$. We may
  assume that $X$ is a complex of projective objects with $X^n=0$ for
  $n\gg 0$ and $X$ acyclic in degrees $n\le p$ for some integer
  $p$. Truncating $X$ in degree $p$ gives an exact triangle $X'\to
  X\to X''\to X'[1]$ such that $X'$ is a bounded complex of projective
  objects and $X''$ is acyclic in all degrees different from
  $p$. Clearly, $X'$ belongs to $\D^b(\C)$, and $X''$ is isomorphic to a shift of 
  an object from $\C$. It follows that $X$ is isomorphic to an object
  in $\D^b(\C)$.
\end{proof}

\begin{rem}\label{rem:stable}
As noted earlier thick subcategories of the stable derived category correspond to thick subcategories of $\D^b(\A)$ containing the projective objects. It follows that the theorem gives a bijection between thick subcategories of the stable derived category and thick subcategories of $\A$ containing all projective objects.
\end{rem}

\begin{rem}\label{rem:thick}
  In the theorem and its proof, one can replace the subcategory
  $\Proj\A$ by any thick subcategory $\P$ of $\A$ having the property
  that each object $X\in\A$ admits an admissible epimorphism $P\to X$
  with $P\in\P$.
\end{rem}

\begin{rem}
A full additive subcategory $\C$ of $\A$ is thick and contains all
projective objects iff it is closed under taking extensions, direct
summands, cosyzygies and syzygies.
\end{rem}

\section{Applications and comments}
In this section we make some brief remarks concerning applications of the bijection of the last section. In particular, we extend the results of Schoutens and Takahashi as promised in the introduction.

\subsection{Classification theorems}
We will explain how to use the main theorem to obtain two classification results, one new and one known, for thick subcategories of abelian categories.

We first use Theorem~\ref{th:thick} to give a classification of the thick subcategories of $\mod A$, which contain $A$, when $A$ is a local complete intersection ring; this extends \cite[Theorem~4.6~(1)]{Tak:2010}. In order to state the classification we need to introduce a hypersurface associated to $A$.

Let $A$ be a local \emph{complete intersection} i.e., there is a regular local ring $B$ and a surjection $B\to A$ with kernel generated by a regular sequence. Set $\bbX = \Spec A$, $\mathbb{T} = \Spec B$, $\mathcal{E} = \str_{\mathbb{T}}^{\oplus c}$, and $t = (b_1,\ldots,b_c)$ where the $b_i$ form a regular sequence generating the kernel of $B\to A$. Let $\mathbb{Y}$ be the hypersurface defined by the section $\Sigma_{i=1}^c b_i x_i$ of $\str_{\mathbb{P}^{c-1}_B}(1)$ where the $x_i$ form a basis for the free $B$-module $H^0(\mathbb{P}^{c-1}_B, \str_{\mathbb{P}^{c-1}_B}(1))$. In summary we are concerned with the following commutative diagram
\begin{displaymath}
\xymatrix{
\mathbb{P}^{c-1}_A \ar[r]^-i \ar[d]_p & \mathbb{Y} \ar[r]^-u & \mathbb{P}^{c-1}_B \ar[d]^{q} \\
\bbX \ar[rr]^-j && \mathbb{T}.
}
\end{displaymath}

The following corollary of Theorem~\ref{th:thick}, based upon \cite[Corollary~10.5]{Ste}, shows that the hypersurface $\mathbb{Y}$ controls the thick subcategories of $\mod A$ which contain the projectives.

\begin{cor}
Let $A$ be a local complete intersection and let $\mathbb{Y}$ be the hypersurface as defined above. Then there is an order preserving bijection
\begin{displaymath}
\left\{ \begin{array}{c}
\text{specialisation closed} \\ \text{subsets of}\; \Sing \mathbb{Y}
\end{array} \right\}
\xymatrix{ \ar[r]<1ex> \ar@{<-}[r]<-1ex> &} \left\{
\begin{array}{c}
\text{thick subcategories of} \\ \mod A \; \text{containing}\; A
\end{array} \right\},
\end{displaymath}
where $\Sing \mathbb{Y}$ denotes the set of $y\in \mathbb{Y}$ such that the local ring $\str_{\mathbb{Y},y}$ is not regular.
\end{cor}
\begin{proof}
By Theorem \ref{th:thick} there is a bijection between thick subcategories of $\mod A$ containing $A$ and thick subcategories of $\D^b(\mod A)$ containing $A$. As stated in Remark \ref{rem:stable} the thick subcategories of $\D^b(\mod A)$ containing $A$ are in bijection with the thick subcategories of the stable derived category $\D^b(\mod A)/\D^b(\proj A)$. The result now follows from the classification of thick subcategories of the stable derived category given in \cite[Corollary~10.5]{Ste}.
\end{proof}

In a similar vein we can apply the theorem to the work of Benson, Carlson, and Rickard \cite{BCR} on stable categories in modular representation theory. Let $G$ be a finite group and let $k$ be a field whose characteristic divides the order of $G$. We denote the category of finite dimensional representations of $kG$ by $\mod kG$ and by $\stmod kG$ its stable category which is obtained by annihilating all projective modules. We say a subcategory $\C$ of $\mod kG$ is a \emph{thick tensor ideal} if it is a thick subcategory which is closed under tensoring with arbitrary objects of $\mod kG$. Applying Theorem~\ref{th:thick} to \cite[Theorem~3.4]{BCR} immediately recovers the following known result.

\begin{thm}
Let $k$ and $G$ be as above. Then there is an order preserving bijection
\begin{displaymath}
\left\{ \begin{array}{c}
\text{specialisation closed} \\ \text{subsets of}\; \agProj H^*(G,k)
\end{array} \right\}
\xymatrix{ \ar[r]<1ex> \ar@{<-}[r]<-1ex> &} \left\{
\begin{array}{c}
\text{thick tensor ideals of} \\ \mod kG \; \text{containing}\; kG
\end{array} \right\}
\end{displaymath}
where $H^*(G,k)$ denotes the cohomology of $G$ with coefficients in $k$.
\end{thm}

Although this is implicit in the work of Benson, Carlson, and Rickard it does not seem to be explicitly stated in the literature. One should also compare this with the analogue for the category of all $kG$-modules which appears as Theorem 10.4 in \cite{BIK}.

\subsection{Generating categories of coherent sheaves on schemes}

The \emph{singular locus} $\Sing A$ of a commutative noetherian ring
$A$ consists of all prime ideals $\frp$ such that the localisation
$A_\frp$ is not regular. More generally the \emph{singular locus} $\Sing \bbX$ of a scheme $\bbX$ consists of those points $x\in \bbX$ such that the local ring $\str_{\bbX,x}$ is not regular. We give a new proof of the following result.

\begin{thm}[Schoutens \cite{Sch}]\label{th:Sch}
  Let $A$ be a commutative noetherian ring. Then the smallest thick
  subcategory of the category of $A$-modules containing $A$ and
  $A/\frp$ for all $\frp$ in the singular locus of $A$ equals the
  category of noetherian $A$-modules.
\end{thm}

As a corollary we will prove a generalisation of this result for schemes having an ample family of line bundles. Keeping this generalisation in mind, let us fix a separated noetherian scheme $\bbX$ and let $\{\mathcal{L}_i \; \vert \; 1\leq i \leq n\}$ be an ample family of line bundles on $\bbX$. Recall that a family of line bundles $\{\mathcal{L}_i \; \vert \; 1\leq i \leq n\}$ on $\bbX$ is \emph{ample} if there is a family of sections $f\in H^0(\bbX, \mathcal{L}_i^{\otimes m})$ with $1\leq i \leq n$ and $m>0$ such that the $\bbX_f = \{x\in \bbX\; \vert\; f_x \notin \frm_x\mathcal{L}^{\otimes m}_x\}$, where $\frm_x$ is the maximal ideal of $\str_{\bbX,x}$, form an open affine cover of $\bbX$. Of course one can just keep in mind the example of an affine scheme $\bbX = \Spec A$ and take $\{\str_\bbX\}$ for the ample family. We denote by $\coh \bbX$ the abelian category of coherent sheaves of $\str_\bbX$-modules.

The proof passes through the homotopy category of injective sheaves of $\str_\bbX$-modules. Let
$\K(\bbX)$ denote the homotopy category of complexes of quasi-coherent $\str_\bbX$-modules and
$\K(\Inj \bbX)$ the full subcategory consisting of complexes of injective quasi-coherent
$\str_\bbX$-modules. We identify each quasi-coherent sheaf with the corresponding
complex in $\K(\bbX)$ concentrated in degree zero.

\begin{lem}\label{le:KInj}
  Let $\bbX$ be a noetherian scheme as above and $\C$ a subcategory
  of $\coh \bbX$ containing the sheaves $\{\mathcal{L}_i^{\otimes m} \; \vert \; 1\leq i \leq n, m\in \bbZ\}$. Suppose that any complex $Y$ of injective quasi-coherent
  $\str_\bbX$-modules is nullhomotopic provided that
\[\Hom_{\K(\bbX)}(X,Y[n])=0 \qquad\text{for all}\quad X\in\C,\,n\in\bbZ.\]
Then the smallest thick subcategory of coherent sheaves containing $\C$
is $\coh \bbX$.
\end{lem}
\begin{proof}
  The functor $\bfi\colon \coh \bbX\to \K(\Inj \bbX)$ taking a sheaf to its
  injective resolution extends to an equivalence $\D^b(\coh
  \str_\bbX)\xto{\sim}\K(\Inj \bbX)^c$ onto the full subcategory of compact
  objects of $\K(\Inj \bbX)$, by \cite[Proposition~2.3]{Kra}.  Note
  that \[\Hom_{\K(\bbX)}(\bfi X,Y)\xto{\sim} \Hom_{\K(\bbX)}(X,Y) \qquad\text{for all}\quad
  Y\in\K(\Inj \bbX),\] by \cite[Lemma~2.1]{Kra}.  The assumption on $\C$
  implies that the thick subcategory of $\K(\Inj \bbX)$ generated by
  $\bfi(\C)$ equals $\K(\Inj \bbX)^c$. This follows from a standard
  argument involving Bousfield localisation; see
  \cite[Lemma~2.2]{Nee1992}. Thus the correspondence in
  Theorem~\ref{th:thick}, together with Remark~\ref{rem:thick}, implies that $\C$ generates $\coh \bbX$.
\end{proof}

\begin{proof}[Proof of Theorem~\ref{th:Sch}]
  In view of Lemma~\ref{le:KInj}, it suffices to show that each
  complex $X$ of injective $A$-modules is nullhomotopic provided that
\[\Hom_{\K(A)}(A,X[n])=0 \quad\text{and}\quad \Hom_{\K(A)}(A/\frp,X[n])=0\]
for all $n\in\bbZ$ and $\frp\in\Sing A$.  

Thus we fix a complex $X$ satisfying these conditions. The first one
implies that $X$ is acyclic. We may assume that $X$ is
homotopically minimal, that is, there is no non-zero direct summand of
$X$ which is nullhomotopic; see \cite[Appendix~B]{Kra}. 
This means that for each $n\in\bbZ$ the truncated complex
\[X^n\to X^{n+1}\to X^{n+2}\to\cdots\] yields a minimal injective
resolution of $Z^n(X)$. Localising at a prime ideal $\frp$ preserves
this property.  If $\frp\not\in\Sing A$, then $Z^n(X)_\frp$ is
injective, and therefore $Z^{n+1}(X)_\frp=0$.  Thus $X_\frp=0$. Now
consider for each $n\in\bbZ$ the exact sequence
\[0\to Z^{n}(X)\to X^{n}\to Z^{n+1}(X)\to 0\] of modules supported on
$\Sing A$.  If this sequence does not split, one finds a finitely
generated submodule $U$ of $Z^{n+1}(X) $ isomorphic to $A/\frp$ for
some $\frp\in\Sing A$ such that $\Ext^1_A(U,Z^{n}(X))\neq 0$.  This
follows from Baer's criterion, and one uses that the prime ideals in
$\Sing A$ form a specialisation closed subset of $\Spec A$.  Now
observe that
\[\Hom_{\K(A)}(-,X[n+1])\cong\Ext^1_\A(-,Z^{n}(X)).\]
Thus the assumption on $X$ implies that it is nullhomotopic.
\end{proof}

\begin{cor}\label{th:Sch2}
Let $\bbX$ be a separated noetherian scheme and let $\{\mathcal{L}_i \; \vert \; 1\leq i \leq n\}$ be an ample family of line bundles on $\bbX$. Then the smallest thick subcategory of $\coh \bbX$ containing the set of coherent $\str_\bbX$-modules
\begin{displaymath}
\mathcal{S} = \{\mathcal{L}_i^{\otimes m} \otimes_{\str_\bbX} \str_{\mathcal{V}(x)}^\a \; \vert \; 1\leq i \leq n, m \in \bbZ, x\in \Sing \bbX, \a\in \{0,1\}\}
\end{displaymath}
where $\str_{\mathcal{V}(x)}$ denotes the structure sheaf of the closed subset $\mathcal{V}(x)$ endowed with the reduced induced scheme structure and $\str_{\mathcal{V}(x)}^0 = \str_\bbX$, is the whole category of coherent $\str_\bbX$-modules.
\end{cor}
\begin{proof}
We proceed essentially as in the theorem i.e., we show that if $X$ is a complex in $\K(\Inj \bbX)$ which is not nullhomotopic then some object of $\mathcal{S}$ maps to a suspension of $X$. So let us fix such a complex $X$. We may assume that $X$ is acyclic as the tensor powers of the sheaves in the ample family generate the derived category (see \cite[Example~1.11]{NeeGD}) and thus detect any complex having non-zero cohomology.

Since $X$ is not nullhomotopic there exists a $j\in \bbZ$ such that $Z^j(X)$ is not injective. Let $f\in H^0(\bbX, \mathcal{L}_i^{\otimes m})$ be a section such that $\bbX_f = \{x\in \bbX\; \vert\; f_x \notin \frm_x\mathcal{L}^{\otimes m}_x\}$ is an open affine on which $Z^j(X)\vert_{\bbX_f}$ is not injective; we can find such an $f$ by ampleness of the family of line bundles. In particular, $X\vert_{\bbX_f}$ is non-zero in $\K(\Inj \bbX_f)$. Hence, by the proof of the theorem, there is an $x\in \Sing \bbX$ and a non-zero map $g~\in~\Hom_{\K(\bbX_f)}(\str_{\mathcal{V}(x)}\vert_{\bbX_f}, X\vert_{\bbX_f}[j+1])$ for some integer $j$. By \cite[Lemma~III.5.14]{Hartshorne} there exists an $m'\geq 0$ so that $f^{m'}g$ lifts to a morphism
\begin{displaymath}
\tilde{g}\colon \str_{\mathcal{V}(x)} \to X \otimes_{\str_\bbX} \mathcal{L}_i^{\otimes m'}[j+1].
\end{displaymath}
The map $\tilde{g}$ is not nullhomotopic as if it were this would yield a nullhomotopy for $f^{m'}g$ and thus $g$. It just remains to note that, by adjunction, the map $\tilde{g}$ is equivalent to a non-zero morphism $\str_{\mathcal{V}(x)} \otimes_{\str_\bbX} \mathcal{L}_i^{\otimes -m'} \to X[j+1]$ witnessing the fact that $X$ is not nullhomotopic.
\end{proof}

\subsection{Strong generators}

Strong generators of triangulated categories were introduced by Bondal
and Van den Bergh \cite[\S2.2]{BovB}.  There seems to be no obvious
analogue for exact categories. So it would be interesting to translate
the following result into a statement about abelian categories, using the
bijection from Theorem~\ref{th:thick}.

\begin{thm}[Oppermann--Stovicek \cite{OpSt}]
  Let $k$ be a field and $\A$ a $k$-linear abelian category which is
  Hom-finite and has enough projective objects. If $\D$ is a thick
  subcategory of $\D^b(\A)$ which contains all projective objects and
  admits a strong generator, then $\D=\D^b(\A)$.\qed
\end{thm}


\begin{thebibliography}{99}
%
\bibitem{BaSc} P. Balmer\ and\ M. Schlichting, Idempotent completion
  of triangulated categories, J. Algebra {\bf 236} (2001), no.~2,
  819--834.
%
\bibitem{BeKr} A. Beligiannis\ and\ H. Krause, Thick subcategories and
  virtually Gorenstein algebras, Illinois J. Math. {\bf 52} (2008),
  no.~2, 551--562.
%
\bibitem{BCR} D. J. Benson, J. F. Carlson\ and\ J. Rickard, Thick
  subcategories of the stable module category, Fund. Math. {\bf 153}
  (1997), no.~1, 59--80.
%
\bibitem{BIK} D.~J. Benson, S.~B. Iyengar, and H.~Krause, Stratifying
  modular representations of finite groups, Ann. of Math. (1)
  \textbf{174} (2011); to appear, arxiv:0810.1339.
%
\bibitem{BovB} A. Bondal\ and\ M. Van den Bergh, Generators and
  representability of functors in commutative and noncommutative
  geometry, Mosc. Math. J. {\bf 3} (2003), no.~1, 1--36, 258.
%
\bibitem{Buc} R.-O. Buchweitz, Maximal Cohen-Macaulay modules and
  Tate-cohomology over Gorenstein rings, Unpublished manuscript
  (1987), 155 pp.
%
\bibitem{Hartshorne} R. Hartshorne, Algebraic geometry, Graduate 
Texts in Mathematics, No. 52. Springer-Verlag
%
\bibitem{Kra} H. Krause, The stable derived category of a Noetherian
  scheme, Compos. Math. {\bf 141} (2005), no.~5, 1128--1162.
%
\bibitem{Nee} A. Neeman, The derived category of an exact category,
  J. Algebra {\bf 135} (1990), no.~2, 388--394.
%
\bibitem{Nee1992} A. Neeman, The connection between the $K$-theory
  localization theorem of Thomason, Trobaugh and Yao and the smashing
  subcategories of Bousfield and Ravenel, Ann. Sci. \'Ecole
  Norm. Sup. (4) {\bf 25} (1992), no.~5, 547--566.
%
\bibitem{NeeGD} A. Neeman, The Grothendieck duality theorem via Bousfield's techniques and Brown
   representability, J. Amer. Math. Soc. {\bf 9} (1996), no.~1 205--236
%
\bibitem{OpSt} S. Oppermann and  J. Stovicek, Generating the bounded
  derived category and perfect ghosts, to appear in Bull. Lond. Math. Soc.

\bibitem{Orl} D. O. Orlov, Triangulated categories of singularities
  and D-branes in Landau-Ginzburg models, Tr. Mat. Inst. Steklova {\bf
    246} (2004), Algebr. Geom. Metody, Svyazi i Prilozh., 240--262;
  translation in Proc. Steklov Inst. Math. {\bf 2004}, no.~3 (246),
  227--248.
%
\bibitem{Qui} D. Quillen, Higher algebraic $K$-theory. I, in {\it
    Algebraic $K$-theory, I: Higher $K$-theories (Proc. Conf.,
    Battelle Memorial Inst., Seattle, Wash., 1972)}, 85--147. Lecture
  Notes in Math., 341, Springer, Berlin, 1973.
%
\bibitem{Sch} H. Schoutens, Projective dimension and the singular
  locus, Comm. Algebra {\bf 31} (2003), no.~1, 217--239.
%
\bibitem{Ste} G. Stevenson, Subcategories of singularity categories via tensor actions, 	arXiv:1105.4698.
%
%
\bibitem{Tak} R. Takahashi, Classifying thick subcategories of the stable category of Cohen-Macaulay modules, Adv. Math. {\bf 225} (2010), no.~4, 2076--2116.
%
\bibitem{Tak:2010} R. Takahashi, Thick subcategories over Gorenstein local rings that are locally hypersurfaces on the punctured spectra, J. Math. Soc. Japan; to appear, arXiv:1109.3120
%
\bibitem{Ver} J.-L. Verdier, Des cat\'egories d\'eriv\'ees des cat\'egories ab\'eliennes, Ast\'erisque No. 239 (1996), {\rm xii}+253 pp. (1997).

\end{thebibliography}
\end{document}